\documentclass[11pt]{article}

\usepackage{preamble}
\usepackage[nottoc]{tocbibind}
\numberwithin{equation}{section}

\title{Hilbert spaces admit no finitary discrete imaginaries}
\author{Ruiyuan Chen and Isabel Trindade}
\date{}

\begin{document}

\maketitle

\begin{abstract}
We prove that every functor from the category of Hilbert spaces and linear isometric embeddings to the category of sets which preserves directed colimits must be essentially constant on all infinite-dimensional spaces.
In other words, every finitary set-valued imaginary over the theory of Hilbert spaces, in a broad signature-independent sense, must be essentially trivial.
This extends a result and answers a question by Lieberman--Rosický--Vasey, who showed that no such functor on the supercategory of Hilbert spaces and injective linear contractions can be faithful.
\let\thefootnote=\relax
\end{abstract}

\section{Introduction}

Traditional model theory formulated for first-order logic is well-suited to studying properties of finitary mathematical structures, such as graphs, linear orders, groups, rings, Boolean algebras, etc.
It has long been known that natural classes of structures that arise in analysis and topology, such as Banach spaces, operator algebras, measure-preserving group actions, etc., cannot be adequately described as first-order structures.
Intuitively speaking, this is due to the infinitary or ``continuous'' aspects of the structure of interest (e.g., information about convergent sequences or $L^\infty$ functions).
The \emph{continuous first-order logic} developed in the last few decades (see \cite{BBHU:ctslog}) has been highly successful at adapting model-theoretic ideas to studying such analytically-flavored structures.

In trying to formalize the intuition that structures such as Banach spaces and $C^*$-algebras are ``intrinsically continuous'', hence cannot be described in discrete first-order logic, a subtlety arises: such a result can only be meaningful if formulated in a signature-independent sense.
For example, $C^*$-algebras possess a robust model theory in both continuous first-order logic (see \cite{FHLRTVW:cstar}) and infinitary equational logic (see \cite{manes:algthy}), but only by taking the ``underlying set'' to be the unit ball, rather than the underlying vector space.
\emph{A priori} it is conceivable that yet a different choice of underlying set and signature could even yield a discrete first-order axiomatization.

Given an abstract category $\!C$, one can regard an arbitrary functor $U : \!C -> \Set$ to the category of sets as a possible assignment of ``underlying set'' to the objects in $\!C$.
Note that if $\!C$ in fact consists of some concretely defined structures, then by general principles,
\unskip\footnote{By the Yoneda lemma \cite[3.10]{kelly:enriched}, $U$ is a colimit, i.e., quotient of coproduct, of representables $\!C(K,-) \subseteq (-)^K$.}
every such $U$ admits a syntactic description: $U$ maps each structure $A \in \!C$ to a definable quotient of a disjoint union of definable sets of (possibly infinite) tuples.
Such a syntactic specification of a uniform assignment of a set to each structure is called an \defn{imaginary sort}.
When the imaginary sort is finitary (defined in first-order logic, and involving only finite tuples and disjoint unions), then the resulting functor $U : \!C -> \Set$ will preserve directed colimits.
This also holds more generally for imaginaries in the $\Sigma_1$ fragment of the infinitary logic $\@L_{\infty\omega}$, also known (in relation to topos theory) as \emph{geometric logic}.
See \cite[\S4.3]{hodges}, \cite{makkai-reyes}, \cite{johnstone:elephant} for background on imaginaries in finitary and infinitary logic.

\addpenalty{-300}
In \cite{LRV}, it was shown that

\begin{theorem}[Lieberman--Rosický--Vasey]
\label{thm:lrv}
There is no faithful directed-colimit-preserving functor $U$ from the category $\Hilbm$ of Hilbert spaces and injective linear contractions to the category of sets.
\end{theorem}

In light of the discussion above, this means that there is no discrete first-order (or even geometric) axiomatization of the category $\Hilbm$, even allowing for unconventional choices of ``underlying set'', thereby showing that in a strong sense, Hilbert spaces are ``intrinsically continuous'' structures.
Using this result, the authors also showed how to easily deduce that the categories of complete metric spaces, Banach spaces, and (commutative) $C^*$-algebras, with corresponding injective $1$-Lipschitz homomorphisms, also do not admit faithful directed-colimit-preserving functors to $\Set$.

The main result of this paper is a significant strengthening of \cref{thm:lrv}:

\begin{theorem}[\cref{thm:hilbr-Uconstant}]
\label{thm:intro-hilbr-Uconstant}
Every directed-colimit-preserving functor $U$ from the category $\Hilbr$ of Hilbert spaces and linear isometric embeddings to the category of sets is essentially constant on all infinite-dimensional spaces.
\end{theorem}

A functor is \defn{essentially constant} if it is naturally isomorphic to a constant functor.
In syntactic terms, whereas \cref{thm:lrv} says that no finitary discrete-set-valued imaginary can capture all of the structure on a Hilbert space, \cref{thm:intro-hilbr-Uconstant} says that no such imaginary can capture any structure whatsoever on infinite-dimensional spaces; it must be a constant disjoint union of singletons.
We stress however that, like \cref{thm:lrv}, \cref{thm:intro-hilbr-Uconstant} really concerns $\Hilbr$ as an \emph{abstract} category, and is independent of any prior representation of Hilbert spaces as concrete structures.

Note also that \cref{thm:intro-hilbr-Uconstant} concerns the category $\Hilbr$ of Hilbert spaces with model-theoretic embeddings as morphisms (which are elementary embeddings in continuous logic between infinite-dimensional spaces \cite[\S15]{BBHU:ctslog}); whereas \cref{thm:lrv} concerns the supercategory $\Hilbm$ of Hilbert spaces with injective homomorphisms, which itself lacks all directed colimits.
\unskip\footnote{This is contrary to a claim in \cite[\S3, first paragraph]{LRV}; indeed, note that the colimit of the sequence $\#C -> \#C -> \#C -> \dotsb$ where each morphism scales by $\frac12$ does not exist in $\Hilbm$.
The proof in \cite{LRV} nonetheless works to show that there is no faithful functor $U : \Hilbm -> \Set$ preserving all \emph{existing} directed colimits.}
Our proof of \cref{thm:intro-hilbr-Uconstant} is broadly based on the ideas in the proof of \cref{thm:lrv} from \cite{LRV}; the main new technical difficulties all arise from replacing $\Hilbm$ with $\Hilbr$.
In particular, we obtain the analogue of \cref{thm:lrv} for $\Hilbr$, which answers a question from \cite[Remark~20]{LRV}.

It is worth remarking that \cref{thm:intro-hilbr-Uconstant} continues to hold if the functor $U$ is defined only on the full subcategory $\Hilbr^{\ge\aleph_0} \subseteq \Hilbr$ of infinite-dimensional spaces to begin with; see \cref{thm:hilbrk-Uconstant}.

In \cref{sec:misc}, we also use \cref{thm:intro-hilbr-Uconstant} to deduce easy consequences for some related categories:

\begin{corollary}[\cref{thm:hilb-Uconstant}]
Every directed-colimit-preserving functor from the category $\Hilb$ of Hilbert spaces and all linear contractions to the category of sets is essentially constant.
\end{corollary}

\begin{corollary}[\cref{thm:hilbm-Uconstant}]
Every directed-colimit-preserving functor $U : \Hilbm -> \Set$ is essentially constant on all infinite-dimensional spaces.
\end{corollary}

\begin{corollary}[\cref{thm:metr-banr}]
There is no faithful directed-colimit-preserving functor $U$ from the category $\Metr$ of complete metric spaces (of diameter $\le 1$) and isometric embeddings, or from the category $\Banr$ of Banach spaces and linear isometric embeddings, to the category of sets.
\end{corollary}

\paragraph*{Acknowledgments}

The authors were supported by NSF grant DMS-2224709.

\section{Preliminaries}
\label{sec:prelim}

We assume familiarity with basic category theory, including the notions of functor, natural transformation, monomorphism, and colimit; see \cite{maclane:cats}.
A \defn{concrete category} means a category $\!C$ equipped with an arbitrary functor $U : \!C -> \Set$.
A \defn{subobject} of an object $A$ in a category $\!C$ is an equivalence class of monomorphisms $i : A_0 -> A$ into $A$ (two such monomorphisms being equivalent iff they fit into a commutative triangle).
We adopt the usual abuse of terminology regarding subobjects by treating them as if they were (concrete) subsets: we write $A_0 \subseteq A$ to denote (the domain of a generic monomorphism representing) a subobject; we call the monomorphism $i : A_0 -> A$ the \defn{inclusion}; etc.
(Note however that for a concrete category $U : \!C -> \Set$, $U$ is not required to preserve monomorphisms.)

Throughout this paper, a \defn{Hilbert space} will mean a norm-complete inner product space over the field $\#F := \#C$ or $\#R$; all arguments will work equally well for both choices of $\#F$.
In the complex case, we use the convention that inner products are linear in the first variable and conjugate-linear in the second.
For a set $X$, $\ell^2(X)$ denotes the Hilbert space of square-summable functions $X -> \#F$; when $X = \#N$, we also write $\ell^2 := \ell^2(\#N)$.

We let $\Hilb$ denote the category of Hilbert spaces and \defn{linear contractions}, meaning linear maps of operator norm $\le 1$.
We let $\Hilbm \subseteq \Hilb$ denote the subcategory with the same objects but only the monomorphisms, i.e., injective linear contractions, and $\Hilbr \subseteq \Hilbm$ denote the further subcategory consisting of only the regular monomorphisms, i.e., \defn{linear isometries}, meaning linear maps preserving the norm (or equivalently the inner product).

\section{Hilbert spaces and linear isometries}

\subsection{Supports of elements}

The following notion is derived from \cite{henry2020abstractelementaryclassnonaxiomatizable}:

\begin{definition}
\label{def:supp}
Let $U : \!C -> \Set$ be a concrete category, $A \in \!C$ be an object, $A_0 \subseteq A$ be a subobject with inclusion $i : A_0 `-> A$, and $x \in UA$.
We say $x$ is \defn{supported on $A_0$ in $A$} if whenever two morphisms $f, g : A \rightrightarrows B$ in $\!C$ agree on $A_0$ (meaning $fi = gi$), then $Uf(x) = Ug(x)$.
\end{definition}

\begin{remark}
\label{rmk:supp-trivial}
If $\!C$ has an initial object $0$ with a monomorphism to $A \in \!C$, then to say that $x \in UA$ is supported on $0$ means that $Uf(x) = Ug(x)$ for every two morphisms $f, g : A \rightrightarrows B$.
We say that $x$ has \defn{trivial support} in this case.

Thus, to say that every $x \in UA$ has trivial support means that $Uf = Ug$ for every $f, g : A \rightrightarrows B$.
\end{remark}

\begin{remark}
\label{rmk:supp-super}
Note that the definition of ``supported on $A_0$ in $A$'' \emph{a priori} depends on the object $A$.

For subobjects $A_0 \subseteq A \subseteq B$ with inclusions $i : A_0 `-> A$ and $f : A `-> B$, if $x \in UA$ is supported on $A_0 \subseteq A$, then $Uf(x) \in UB$ is supported on $A_0 \subseteq B$.
(Given $g, h : B \rightrightarrows C$, if $gfi = hfi$, then $Ug(Uf(x)) = U(gf)(x) = U(hf)(x) = Uh(Uf(x))$ since $x$ is supported on $A_0$.)

Thus, if $x$ has trivial support in $A$, then $x$ also has trivial support in any $B \supseteq A$.
\end{remark}

Fix now a directed-colimit-preserving functor $U : \Hilbr -> \Set$.
Our goal is to show that $U$ is essentially constant on infinite-dimensional spaces.
By \cref{rmk:supp-trivial}, this entails showing that every $x \in UA$ for infinite-dimensional $A$ has trivial support.
As in \cite{LRV}, the crucial technical ingredient is to show that ``supports intersect'':

\begin{lemma}
\label{thm:hilbr-supp-int}
If $A$ is an infinite-dimensional Hilbert space, $x \in UA$, and $A_0, A_1$ finite-dimensional subspaces of $A$ such that $x$ is supported on $A_0$ and on $A_1$, then $x$ is supported on $A_0\cap A_1$.
\end{lemma}

To prove this, we need the following elementary linear algebra fact.
Recall from \cref{sec:prelim} that $\#F$ denotes the scalar field $\#C$ or $\#R$.

\begin{lemma}
Let $\vec{u}, \vec{v}, \vec{w} \in \#F^3$ be unit vectors such that $\vec{u}, \vec{v}$ are not parallel.
Then there is a finite composition of linear isometric automorphisms $\#F^3 -> \#F^3$ that fix either $\vec{u}$ or $\vec{v}$, which takes $\vec{u}$ to $\vec{w}$.
\end{lemma}
\begin{proof}
Since $\vec{u}, \vec{v}$ are not parallel, we have $\abs{\ang{\vec{u}, \vec{v}}} < 1 = \cos(0)$ by the Cauchy--Schwarz inequality.
We induct on the least $n \in \#N$ such that $\abs{\ang{\vec{u}, \vec{v}}} \le \cos(\pi/2^{n+1})$ (i.e., when $\#F = \#R$, the least $n$ such that the acute angle between the lines spanned by $\vec{u}, \vec{v}$ is $\ge 90^\circ/2^n$).

First, suppose $n = 0$, i.e., $\vec{u} \perp \vec{v}$.
By conjugating $\#F^3$ by a suitable linear isometry, we may assume $\vec{u} = \vec{e}_1$ and $\vec{v} = \vec{e}_2$ are the standard basis vectors.
Let $\vec{w} = (a,b,c) \in \#F^3$.
Then there is an isometry $\#F^3 -> \#F^3$ fixing $\vec{v}$ taking $\vec{u} |-> (a,0,d)$ for some $d \in \#F$ (e.g., $d = \sqrt{1-\abs{a}^2}$), followed by an isometry fixing $\vec{u}$ taking $(a,0,d) |-> (a,b,c) = \vec{w}$.

Now suppose $n$ is such that the claim is known for any $\vec{u}, \vec{v}$ with
$\abs{\ang{\vec{u}, \vec{v}}} \le \cos(\pi/2^{n+1})$;
we must show the same when $n$ is replaced with $n+1$.
By conjugating as above, we may assume $\vec{v} = \vec{e}_1$ and
$\vec{u} = (\ang{\vec{u},\vec{v}},a,0)$
where
$a := \sqrt{1-\abs{\ang{\vec{u},\vec{v}}}^2}$.
Observe that for any $\theta \in \#R$, there is an isometry $f_\theta : \#F^3 -> \#F^3$ fixing $\vec{v}$ and taking
$\vec{u} |-> (\ang{\vec{u},\vec{v}},a\cos(\theta),a\sin(\theta)) =: \vec{u}_\theta$.
When $\theta = 0$, we have
$\ang{\vec{u},\vec{u}_\theta}
= \ang{\vec{u},\vec{u}}
= 1$;
when $\theta = \pi$, we have
\begin{align*}
\abs{\ang{\vec{u},\vec{u}_\theta}}
&= \abs{\ang{\vec{u},\vec{v}}}^2 - a^2
= 2\abs{\ang{\vec{u},\vec{v}}}^2 - 1 \\
&\le 2\cos(\pi/2^{n+2})^2 - 1
= \cos(\pi/2^{n+1}).
\end{align*}
Thus there is some $0 \le \theta \le \pi$ such that
$\abs{\ang{\vec{u},\vec{u}_\theta}} = \cos(\pi/2^{n+1})$.
By the induction hypothesis, there is a finite composition of isometries fixing either $\vec{u}$ or $\vec{u}_\theta$ and taking $\vec{u} |-> \vec{w}$; the isometries that fix $\vec{u}_\theta$ are a $f_\theta$-conjugate of an isometry that fixes $\vec{u}$, so we are done.
\end{proof}

\begin{proof}[Proof of \cref{thm:hilbr-supp-int}]
We may assume that $A_0 \cap A_1$ is codimension $1$ in both $A_0$ and $A_1$.
Indeed, assuming this special case has been shown, we may deduce the general case as follows.
If $A_0 \subseteq A_1$ or $A_1 \subseteq A_0$, then the claim is trivial.
If $A_0 \cap A_1$ is codimension $1$ in $A_0$ and codimension $n > 0$ in $A_1$, then letting $\vec{v}_0, \dotsc, \vec{v}_{n-1}$ be a basis for $A_1 \cap (A_0 \cap A_1)^\perp$, we have that $x$ is supported on the linear span of $A_0 \cup \set{\vec{v}_0, \dotsc, \vec{v}_{n-2}}$, and on $A_1$, which have intersection of codimension $1$ in each, whence by the special case, $x$ is supported on said intersection, the linear span of $(A_0 \cap A_1) \cup \set{\vec{v}_0, \dotsc, \vec{v}_{n-2}}$;
by induction on $n$, we eventually get that $x$ is supported on $A_0 \cap A_1$.
Finally, in general if $A_0 \cap A_1$ is codimension $m > 0$ in $A_0$, we have that $x$ is supported on $A_0$ and on the span of $A_1$ together with $m-1$ basis vectors in $A_0 \cap (A_0 \cap A_1)^\perp$; by the previous case, we get that $x$ is supported on the span of $A_0 \cap A_1$ together with said $m-1$ vectors; now induct on $m$.

So suppose $A_0$ is spanned by $(A_0 \cap A_1) \cup \set{\vec{u}}$, and $A_1$ by $(A_0 \cap A_1) \cup \set{\vec{v}}$, where $\vec{u} \in A_0 \cap (A_0 \cap A_1)^\perp$ and $\vec{v} \in A_1 \cap (A_0 \cap A_1)^\perp$ are unit vectors.
Let $f, g : A \rightrightarrows B \in \Hilbr$ agree on $A_0 \cap A_1$; we must show that $Uf(x) = Ug(x)$.

First, suppose $g$ is surjective.
Put $\vec{w} := g^{-1}(f(\vec{u}))$; then $\vec{w}$ is a unit vector, such that
\begin{align*}
\vec{w} &\perp A_0 \cap A_1,
\shortintertext{since}
g(\vec{w}) = f(\vec{u}) &\perp f(A_0 \cap A_1) = g(A_0 \cap A_1).
\end{align*}
By the preceding lemma applied to the subspace spanned by $\vec{u}, \vec{v}, \vec{w}$, which is orthogonal to $A_0 \cap A_1$, find linear isometric automorphisms $h_0, \dotsc, h_{k-1} : A \cong A$ each of which fixes either $(A_0 \cap A_1) \cup \vec{u}$ or $(A_0 \cap A_1) \cup \vec{v}$, hence either $A_0$ or $A_1$, such that $(h_0 \circ \dotsb \circ h_{k-1})(\vec{u}) = \vec{w}$.
Then in the sequence
\begin{align*}
g,\;
g \circ h_0,\;
g \circ h_0 \circ h_1,\;
g \circ h_0 \circ h_1 \circ h_2,\;
\dotsc,\;
g \circ h_0 \circ \dotsb \circ h_{k-1},\;
f : A -> B,
\end{align*}
each consecutive pair of maps agrees on either $A_0$ or $A_1$, so $Ug(x) = Uf(x)$ since $x$ is supported on $A_0$ and on $A_1$.
(This argument is a more involved version of \cite[Lemma~14]{LRV}.)

If $g$ is not surjective, note that $f(A) + g(A) \subseteq B$ has the same dimension as $A$.
Let $h : A \cong f(A) + g(A)$ be an arbitrary isometric isomorphism, let $i : f(A) + g(A) `-> B$ be the subspace inclusion, and let $f', g' : A \rightrightarrows f(A) + g(A)$ be the codomain restrictions of $f, g$.
Then $Uf'(x) = Uh(x) = Ug'(x)$ by the surjective case, whence $Uf(x) = Ui(Uf'(x)) = Ui(Ug'(x)) = Ug(x)$.
\end{proof}

\begin{corollary}
\label{thm:hilbr-supp}
For any $A \in \Hilbr$ and $x \in UA$, there is a unique minimal finite-dimensional subspace of $A$ on which $A$ is supported.
We call this the \defn{support of $x$ in $A$}, denoted $\supp_A(x)$.
\end{corollary}
\begin{proof}
Since $U$ preserves directed colimits and $A$ is the directed colimit of its finite-dimensional subspaces $A_0 \subseteq A$, there is some such $A_0$, with inclusion $i : A_0 -> A$, as well as $x_0 \in UA_0$ such that $x = Ui(x_0)$.
In particular, it follows by \cref{rmk:supp-super} that $x$ is supported on $A_0$ (in $A_0$, hence in $A$).
Let $\supp_A(x) \subseteq A$ be a finite-dimensional subspace of minimal dimension on which $x$ is supported; by \cref{thm:hilbr-supp-int}, it is in fact the unique minimal support of $x$ in $A$.
\end{proof}

\begin{remark}
\label{rmk:hilbr-supp-super}
By \cref{rmk:supp-super}, for a linear isometric embedding $f : A `-> B \in \Hilbr$ and $x \in UA$, we have $f(\supp_A(x)) \supseteq \supp_B(Uf(x))$.

Thus for a linear isometric isomorphism $f : A \cong B$, we have $f(\supp_A(x)) = \supp_B(Uf(x))$.
\end{remark}

\subsection{$U$ is essentially constant}

\begin{proposition}
\label{Uconstant-lambda}
There exists an infinite cardinal $\lambda$, such that $Uf = Ug$ for any pair of linear isometric embeddings $f, g : A \rightrightarrows B \in \Hilbr$ where $\dim(A) = \lambda$.
\end{proposition}
\begin{proof}
Following \cite[Theorem~18]{LRV}, pick $\lambda > 2^{\aleph_0}$ of countable cofinality such that $U$ maps a Hilbert space of dimension $\lambda$ to a set of cardinality $\lambda$.
(For the reader's convenience, we sketch an elementary argument for the existence of such $\lambda$ that does not use any accessible category theory.
Let $\lambda_0$ be the supremum of $2^{\aleph_0}$ and the cardinalities of $U(\#F^n)$ for each $n \in \#N$.
Then for any Hilbert space $A$ of dimension $\lambda \ge \lambda_0$, without loss of generality $A = \ell^2(\lambda)$, $A$ is the directed colimit of the subspaces $\ell^2(F)$ for all finite subsets $F \subseteq \lambda$; for each such $F$, we have $\abs{U(\ell^2(F))} \le \lambda_0 \le \lambda$; and there are $\lambda$-many such $F$, hence
$
\abs{U(\ell^2(\lambda))}
= \abs[\big]{\colim_F U(\ell^2(F))}
\le \lambda \cdot \lambda = \lambda.
$
It suffices to take any $\lambda \ge \lambda_0$ of countable cofinality, e.g., $\lambda = \sup \{\lambda_0, \lambda_0^+, \lambda_0^{++}, \dotsc\}$.)

Now let $A$ be a Hilbert space of dimension $\lambda$.
Suppose for contradiction that there were $x \in UA$ with $\supp_A(x) \subseteq A$ a nontrivial finite-dimensional subspace, say of dimension $1 \le n < \aleph_0$.
Observe that $\abs{A} = \abs{\ell^2(\lambda)} = \lambda^{\aleph_0}$, and $A$ is the union of all subspaces $B \subseteq A$ of dimension $n$, each of which has cardinality $2^{\aleph_0} < \lambda < \lambda^{\aleph_0}$; thus there are $\lambda^{\aleph_0}$ distinct subspaces $B \subseteq A$ of dimension $n$.
However, for any such $B \subseteq A$ of dimension $n$, there is a linear isometric automorphism $f : A -> A$ taking $\supp_A(x)$ to $B$, and so there must be $y = Uf(x) \in UA$ such that $\supp_A(y) = B$ (by the preceding remark).
This is a contradiction as $\abs{UA} = \lambda < \lambda^{\aleph_0}$ by Kőnig's theorem.
Thus every $x \in UA$ has trivial support, which yields the desired claim by \cref{rmk:supp-trivial}.
\end{proof}

\begin{corollary}
\label{Uconstant}
$U$ is constant on all pairs of morphisms $f, g : A \rightrightarrows B \in \Hilbr$ where $\dim(A) \ge \lambda$.
\end{corollary}
\begin{proof}
$f, g$ are the directed colimits of their respective restrictions $f_0, g_0$ to all subspaces $A_0 \subseteq A$ of dimension exactly $\lambda$, and $Uf_0 = Ug_0$ for all such $A_0$ by the preceding \namecref{Uconstant-lambda}; thus since $U$ preserves directed colimits, $Uf = Ug$.
\end{proof}

\begin{lemma}
\label{injection}
$U$ maps every $f : A -> B \in \Hilbr$ for $\dim(A) \ge \aleph_0$ to an injection.
\end{lemma}
\begin{proof}
First, suppose $\dim(B) = \dim(A)$. Let $x, y \in UA$, and suppose $Uf(x) = Uf(y) \in UB$.
Since $A$ is infinite-dimensional, it is the directed colimit of all finite-dimensional $A_0 \subseteq A$, and so there are $x_0, y_0 \in UA_0$ for some such $A_0$, with inclusion $i : A_0 -> A$, such that $x = Ui(x_0)$ and $y = Ui(y_0)$.
Since $A, B$ have the same infinite dimension, and $A_0 \subseteq A$ is finite-dimensional, there is an isomorphism $A \cong B$ that takes the inclusion $i : A_0 `-> A$ to the embedding $fi : A_0 `-> A `-> B$.
\begin{equation*}
\begin{tikzcd}
x_0, y_0 \in UA_0 & A_0 \dar[hook,"i"'] \rar[hook,"i"] & A \dar[hook, "f"] & UA \ni x, y \\
x, y \in UA & A \rar[phantom, "\cong"] & B & UB \ni Uf(x) = Uf(y)
\end{tikzcd}
\end{equation*}
By assumption, $x_0, y_0$ map to the same thing in $UB$ along the top-right composite, hence also along the left-bottom, hence also along the left, i.e., $x = Ui(x_0) = Ui(y_0) = y$.

This proves the result for $\dim(B) = \dim(A)$.
If $\dim(B) > \dim(A)$, write $f : A -> B$ as a directed colimit of $f : A -> B_0$, where $B_0 \subseteq B$ contains $\im(f)$ and has the same dimension as $A$.
\end{proof}

\begin{theorem}
\label{thm:hilbr-Uconstant}
The restriction of any directed-colimit-preserving functor $U : \Hilbr -> \Set$ to the full subcategory $\Hilbr^{\ge\aleph_0}$ of infinite-dimensional spaces is naturally isomorphic to a constant functor.
\end{theorem}
\begin{proof}
From \cref{Uconstant}, we know $U$ maps parallel pairs of morphisms between sufficiently high-dimensional Hilbert spaces $\ell^2(\lambda)$ to the same morphism in $\Set$.
For two arbitrary maps $f, g : A \rightrightarrows B \in \Hilbr$ between infinite-dimensional spaces, consider the square
\begin{equation*}
\begin{tikzcd}
A
    \dar[hook]
    \rar[shift left=1, "f"]
    \rar[shift right=1, "g"']
    &
B
    \dar[hook]
\\
A \oplus \ell^2(\lambda)
    \rar[shift left=1, "f \oplus \id"]
    \rar[shift right=1, "g \oplus \id"']
    &
B \oplus \ell^2(\lambda)
\end{tikzcd}
\end{equation*}
Then $U(f \oplus \id) = U(g \oplus \id)$, and the inclusion $B `-> B \oplus \ell^2(\lambda)$ is mapped to an injection by \cref{injection}, hence $Uf = Ug$.
So $U$ is constant on all parallel pairs of morphisms in $\Hilbr^{\ge\aleph_0}$.


This implies that $U$ maps every $f : A -> B \in \Hilbr$ with $\dim(A) \ge \aleph_0$ to a bijection in $\Set$.
To see this: if $\dim(B) = \dim(A)$, then there exists a linear isometric isomorphism $g : A \cong B$; thus $Uf = Ug$ is a bijection.
If $\dim(B) > \dim(A)$, as in the preceding proof, write $f$ as the directed colimit of its codomain restrictions to $B_0 \subseteq B$ of the same dimension as $A$.

We now define a natural isomorphism $\phi : U(\ell^2) \cong U|_{\Hilbr^{\ge\aleph_0}}$, from the constant functor with value $U(\ell^2)$ to the restriction of $U$ to all infinite-dimensional Hilbert spaces $A$.
Let $\phi_A : U(\ell^2) -> UA$ be $Uf$ for any embedding $f : \ell^2 `-> A$.
This is a natural transformation, since for any morphism $g : A -> B \in \Hilbr$, the square
\begin{equation*}
\begin{tikzcd}
U(\ell^2) \dar["\phi_A"'] \rar["\id"] & U(\ell^2) \dar["\phi_B"] \\
UA \rar["Ug"] & UB
\end{tikzcd}
\end{equation*}
commutes, as both composites are $U$ applied to a morphism $\ell^2 -> B$, hence equal by the first part of this proof; and each $\phi_A$ is invertible by the second part of this proof.
\end{proof}

\begin{example}
\label{ex:hilbr-nonconstant}
The conclusion of \cref{thm:hilbr-Uconstant} cannot be strengthened to assert that $U$ itself is essentially constant.
Indeed, let $\-{\#N}$ be the directed-complete linear order $\{0, 1, 2, \dotsc, \infty\}$, and note that we have a directed-colimit-preserving functor $\Hilbr -> \-{\#N}$ taking a finite-dimensional space to its dimension and all infinite-dimensional spaces to $\infty$.
Composing with any directed-colimit-preserving $\-{\#N} -> \Set$ (which is simply the left Kan extension of an arbitrary functor $\#N -> \Set$, e.g., the inclusion on finite ordinals, so taking $\infty |-> \#N$) yields a directed-colimit-preserving $U : \Hilbr -> \Set$, which need not be constant on finite-dimensional spaces.

Similar examples show that the preceding lemma is also false for finite-dimensional $A$, i.e., $U$ need not even preserve monomorphisms between finite-dimensional spaces: e.g., take $\#N -> \Set$ above to be the functor mapping $n$ to the quotient of $\#N$ identifying the first $n$ elements.
\end{example}

\section{Other categories}
\label{sec:misc}

\subsection{Hilbert spaces and linear contractions}

Recall that $\Hilb$ denotes the category of (real or complex) Hilbert spaces and linear \emph{contractions}.
Given a directed-colimit-preserving functor $U : \Hilb -> \Set$, by restricting $U$ to $\Hilbr$ and applying \cref{thm:hilbr-Uconstant}, we of course get that $U$ is essentially constant on all infinite-dimensional spaces as well.
But we can show more: in contrast to \cref{ex:hilbr-nonconstant},

\begin{corollary}
\label{thm:hilb-Uconstant}
Any directed-colimit-preserving functor $U : \Hilb -> \Set$ is essentially constant.
\end{corollary}
\begin{proof}
For any $A \in \Hilb$, let
\begin{alignat*}{3}
i &:& A &-> A \oplus \ell^2
    &&= \text{first coordinate injection}, \\
j &:& A \oplus \ell^2 &-> A \oplus A \oplus \ell^2
    &&= \text{first and third coordinate injection}, \\
k &:& A \oplus \ell^2 &-> A \oplus A \oplus \ell^2
    &&= \text{second and third coordinate injection}, \\
p &:&\; A \oplus A \oplus \ell^2 &-> A
    &&= \text{second coordinate projection};
\end{alignat*}
then
\begin{equation*}
\begin{aligned}
U(0_A)
= U(pji)
&= Up \circ Uj \circ Ui \\
&= Up \circ Uk \circ Ui
    \quad \text{by \cref{thm:hilbr-Uconstant}} \\
&= U(\id_A)
= \id_{UA},
\end{aligned}
\end{equation*}
i.e., $U$ takes the zero map on every $A \in \Hilb$ to the identity.
Now whisker the natural transformations
\begin{equation*}
\begin{tikzcd}
\Hilb \rar[bend left, "\id_\Hilb"] \rar[bend right, "0"'] \rar[phantom, "{\scriptstyle 0}{\Downarrow}{\Uparrow} {\scriptstyle 0}"] &
\Hilb
\end{tikzcd}
\end{equation*}
by $U$ to get that $U$ is naturally isomorphic to the constant functor with value $U(0)$.
\end{proof}

We also note that combining the argument of \cref{thm:hilbr-Uconstant} with \cite{LRV} easily yields the same essential constancy conclusion for the supercategory $\Hilbm$ considered in that paper.
(Here, unlike with $\Hilb$, the argument of \cref{ex:hilbr-nonconstant} shows that we cannot conclude anything for finite-dimensional spaces.)

\begin{corollary}
\label{thm:hilbm-Uconstant}
Any functor $U : \Hilbm -> \Set$ preserving all existing directed colimits in $\Hilbm$ is essentially constant when restricted to the full subcategory $\Hilbm^{\ge\aleph_0}$ of infinite-dimensional spaces.
\end{corollary}
\begin{proof}
By \cite[Lemma~14]{LRV}, there is a well-defined support $\supp_A(x) \subseteq A$ for any infinite-dimensional $A \in \Hilbm$ and $x \in UA$, which is a finite-dimensional subspace of $A$.
By the proofs of \cref{Uconstant-lambda,Uconstant}, it follows that $U$ is constant on all parallel pairs $f, g : A \rightrightarrows B \in \Hilbm$ where $A$ is of sufficiently high dimension $\ge \lambda$ for some infinite cardinal $\lambda$.
By \cref{injection}, we still have that $U$ maps every morphism in the subcategory $\Hilbr^{\ge\aleph_0} \subseteq \Hilbm^{\ge\aleph_0}$ to an injection.
The proof of \cref{thm:hilbr-Uconstant} (which only uses \cref{injection} for a linear isometric embedding $A `-> A \oplus \ell^2(\lambda)$) now applies.
\end{proof}

\subsection{Infinite-dimensional Hilbert spaces}

A somewhat strange aspect of \cref{thm:hilbr-Uconstant} is that the conclusion concerns only the restriction of $U$ to the subcategory $\Hilbr^{\ge\aleph_0}$ of infinite-dimensional Hilbert spaces, yet the definition of $U$ on finite-dimensional spaces was essential to its proof (to make sense of the notion of support).
Nonetheless, we may easily deduce the corresponding result for functors defined on infinite-dimensional spaces:

\begin{corollary}
\label{thm:hilbrk-Uconstant}
For any infinite cardinal $\kappa$, letting $\Hilbr^{\ge\kappa} \subseteq \Hilbr$ be the full subcategory of Hilbert spaces of dimension $\ge\kappa$, every directed-colimit-preserving functor $U : \Hilbr^{\ge\kappa} -> \Set$ is naturally isomorphic to a constant functor.
\end{corollary}
\begin{proof}
Let $V : \Hilbr -> \Set$ take the trivial space to $\emptyset$ and every other space $A$ to $U(\ell^2(\kappa) \otimes A)$, where $\otimes$ denotes the Hilbert space tensor product (see \cite[E3.2.19]{pedersen1989analysis}), with the obvious extension to morphisms.
Then by \cref{thm:hilbr-Uconstant}, $V$ is essentially constant when restricted to $\Hilbr^{\ge\aleph_0}$.

In particular, $U$ takes every morphism $f : A -> B \in \Hilbr^{\ge\kappa}$ which is isomorphic to a morphism of the form $\id_{\ell^2(\kappa)} \otimes g$ for some $g : C -> D \in \Hilbr^{\ge\aleph_0}$ to a bijection $V(g)$; such $f$ are easily seen (by taking an orthonormal basis for $A$, then extending it to one for $B$, and using $\kappa \ge \aleph_0$) to consist of precisely all linear isometries which either are surjective or have $\ge\kappa$-codimensional image.
But every $f : A -> B \in \Hilbr^{\ge\kappa}$ may be composed with some such morphism (e.g., the injection $B `-> B \oplus \ell^2(\kappa)$) to yield another such morphism; thus, $U$ takes every morphism in $\Hilbr^{\ge\kappa}$ to a bijection.

Moreover, $U$ takes every endomorphism $f : A -> A$ which is conjugate to $\id_{\ell^2(\kappa)} \otimes g$ for some $g : C -> C \in \Hilbr^{\ge\aleph_0}$ to the identity.
But every $f : A -> A \in \Hilbr^{\ge\kappa}$ fits into a commutative square
\begin{equation*}
\begin{tikzcd}
\ell^2(\kappa) \otimes A \rar["\id_{\ell^2(\kappa)} \otimes f"] &[1em] \ell^2(\kappa) \otimes A \\
A \uar[hook, "i"] \rar["f"] & A \uar[hook, "i"']
\end{tikzcd}
\end{equation*}
where $i(\vec{v}) := \vec{e}_0 \otimes \vec{v}$;
thus $U$ takes every endomorphism in $\Hilbr^{\ge\kappa}$ to the identity.
As in the proof of \cref{thm:hilbr-Uconstant}, this easily implies that $U$ is constant on any parallel pair $f, g : A \rightrightarrows B \in \Hilbr^{\ge\kappa}$ (by assuming $\dim(A) = \dim(B)$ via codomain restriction and composing with some $h : B \cong A$), which together with the above yields a natural isomorphism $U(\ell^2(\kappa)) \cong U$.
\end{proof}

\subsection{Metric spaces and Banach spaces}

\begin{corollary}
\label{thm:metr-banr}
There is no faithful directed-colimit-preserving functor $U$ from the categories $\Metr$ of complete metric spaces (of bounded diameter $\le1$) and isometric embeddings or $\Banr$ of Banach spaces and linear isometric embeddings to $\Set$, or even from the full subcategories consisting of the non-locally compact spaces.
\end{corollary}
\begin{proof}
Follows from \cref{thm:hilbr-Uconstant} and the faithful forgetful functors $\Hilbr -> \Banr -> \Metr$.
For the last part, use \cref{thm:hilbrk-Uconstant}.
\end{proof}

This strengthens the results in \cite[\S4]{LRV} for the supercategories $\Metm, \Banm$.
However, we do not know if, as for $\Hilbr$ by \cref{thm:hilbr-Uconstant}, every directed-colimit-preserving functor from these categories to $\Set$ must be essentially constant (on non-locally compact spaces, say).

\bibliographystyle{amsalpha}
\def\MR#1{}
\bibliography{references}

\medskip\noindent
Ruiyuan Chen\\
Department of Mathematics\\
University of Florida\\
Gainesville, FL, USA\\
\nolinkurl{ruiyuan.chen@ufl.edu}

\medskip\noindent
Isabel Trindade\\
Institute for Logic, Language, and Computation\\
University of Amsterdam\\
Amsterdam, NL\\
\nolinkurl{isabel.trindade@student.uva.nl}

\end{document}